\documentclass[11pt]{amsart}

\usepackage{amscd,amssymb,amsopn,amsmath,amsthm,mathrsfs,graphics,amsfonts,enumerate,verbatim,calc
}
\usepackage{bbm}
\usepackage[all,cmtip]{xy}
\usepackage[all]{xy}
\usepackage{tikz}
\usepackage{lscape}
\usepackage{enumitem}
\usetikzlibrary{matrix,arrows,decorations.pathmorphing,cd}
\usepackage{hyperref}
\hypersetup{colorlinks=true,linkcolor={blue}}
\usepackage{cleveref}
\usepackage{url}

\usepackage{scalerel}
\usepackage{stackengine,wasysym}
\usetikzlibrary {positioning}
\usetikzlibrary{patterns}
\usetikzlibrary{calc}
\definecolor {processblue}{cmyk}{0.96,0,0,0}

\usepackage{subfigure}

\usepackage{comment} %% To comment out large blocks of text.

\usepackage{ dsfont }

\usepackage{color}

\usepackage[OT2,OT1]{fontenc}
\newcommand\cyr{%
\renewcommand\rmdefault{wncyr}%
\renewcommand\sfdefault{wncyss}%
\renewcommand\encodingdefault{OT2}%
\normalfont
\selectfont}
\DeclareTextFontCommand{\textcyr}{\cyr}

\usepackage{amssymb,amsmath}

\DeclareFontFamily{OT1}{rsfs}{}
\DeclareFontShape{OT1}{rsfs}{n}{it}{<-> rsfs10}{}
\DeclareMathAlphabet{\mathscr}{OT1}{rsfs}{n}{it}

\numberwithin{equation}{section}
\newtheorem{theorem}{Theorem}[section]
\newtheorem{lem}[theorem]{Lemma}
\newtheorem{cor}[theorem]{Corollary}

\newtheorem{prop}[theorem]{Proposition}

\theoremstyle{definition}
\newtheorem{defn}[theorem]{Definition}
\theoremstyle{remark}
\newtheorem{remark}[theorem]{Remark}

\newtheorem{example}[theorem]{Example}

\newcommand{\im}{\operatorname{im}}
\renewcommand{\ker}{\operatorname{ker}}
\newcommand{\tr}{\operatorname{tr}}

\newcommand{\Ext}{\operatorname{Ext}}

\newcommand{\Hom}{\operatorname{Hom}}

\newcommand{\End}{\operatorname{End}}

\newcommand{\coker}{\operatorname{coker}}

\newcommand{\rank}{\ensuremath{\operatorname{rank}}}
\renewcommand{\mod}{\operatorname{mod}}

\newcommand{\Tr}{\operatorname{Tr}}

\newcommand{\m}{\mathfrak{m}}

\newcommand{\w}{\omega}

\renewcommand{\bar}{\overline}

\title[Trace Ideals of Syzygies]{Trace Ideals of Syzygies}

\author[Lindo]{Haydee Lindo}
\email[Haydee Lindo]{hlindo@hmc.edu}
\address{Department of Mathematics \\
320 E. Foothill Blvd. \\
Claremont, CA 91711}

\author[Lyle]{Justin Lyle}
\email[Justin Lyle]{jll0107@auburn.edu}
\urladdr{https://jlyle42.github.io/justinlyle/}
\address{Department of Mathematics and Statistics \\ 305 W Samford Avenue \\
Auburn University, AL 36849}

\author[Maitra]{Sarasij Maitra}
\email[Sarasij Maitra]{sarasij.maitra@villanova.edu}
\urladdr{https://sarasij93.github.io}
\address{Department of Mathematics and Statistics \\ 800 E. Lancaster Avenue \\ Villanova, PA 19085}

\begin{document}


\begin{abstract} 
We study the behavior of trace ideals under taking syzygies. In particular, when $R$ is numerical semigroup ring and $I$ is a homogeneous ideal in $R$, we obtain an upper estimate for $\tr_R(\Omega^1_R(I))$, and we show this estimate is sharp when $I$ is the conductor ideal $\mathfrak{c}_R$ of $R$. Using this result, we characterize the numerical semigroup rings for which $\tr_R(M) \subseteq \tr_R(\Omega^1_R(M))$ holds for every finitely generated $R$-module $M$. In a similar vein, we characterize numerical semigroup rings for which $\tr_R(\Omega^1_R(\mathfrak{c}_R))=\mathfrak{c}_R$.
\end{abstract} 

\maketitle

\section{Introduction}

Let $(R,\m,k)$ be a commutative Noetherian local ring and let $M$ be a finitely generated $R$-module. Then the ideal $\tr_R(M):=\langle f(x) \mid f \in \Hom_R(M,R), x \in M \rangle$ is known as the \emph{trace ideal} of $M$.  
Trace ideals, and trace submodules more generally, were studied classically as images of (bilinear) pairings, either being alluded to or utilized directly in work of Curtis \cite{Cu56}, Morita \cite{Mo58}, and Auslander-Goldman \cite{AG60}, among many others, and later appearing in work of Vasconcelos, M{\"u}ller, and Bass, to detect free summands in the study of reflexive modules \cite{Va68}, and to understand equivalences between  certain module categories \cite{Muller:1974, Bass:1968a}.

Trace ideals have experienced a modern resurgence as more is discovered about the ways in which the structure of $\tr_R(M)$, as an ideal, captures information about the module $M$. In recent years, trace ideals have been used to describe centers of endomorphisms rings, to provide partial progress on longstanding conjectures from the representation theory of Artin algebras, and to codify the failure of the Gorenstein property, to name but a few applications \cite{lindo2017trace,Li22,HH19}. However, despite the renewed attention they have received, numerous fundamental aspects of trace ideals remain poorly understood. 

In this work, we focus on one such aspect: the behaviour of trace ideals  
under taking syzygies. On the one hand, examples exist to show that $\tr_R(M)$ may be much larger, much smaller, or seemingly unrelated to $\tr_R(\Omega^1_R(M))$ (see Examples \ref{incomparable} and \ref{minmultexample}). On the other hand, under certain hypotheses their behavior is tightly connected. For instance, the main results in \cite{BD25} show, if $R$ is a complete intersection or is Golod, that the sequence $\{\tr_R(\Omega^i_R(M))\}$ becomes periodic for $i \gg 0$. However, they also give examples to show that this behavior need not occur for more general $R$. Celikbas-Celikbas-Herzog-Kumashiro show that if $(R,\m,k)$ is a local ring, then $\m \subseteq \Omega^i_R(k)$ for all $i \ge 1$ unless $R$ is either regular or a principal ideal ring with $\m^2 \ne 0$, while it follows from \cite[Theorem 1.2]{lyle2024annihilators} that $\tr(I) \subseteq \tr(\Omega^1_R(I))$ whenever $R$ is a numerical semigroup ring of minimal multiplicity and $\tr_R(I)=\m$. Recent work of Eisenbud-Ficarra-Herzog-Moradi compares the trace of an ideal $I$ with its Fitting ideals, and in particular shows that $\tr_R(\w_R)$ and $I_1(\w_R)$ have the same radical where $\w_R$ is the canonical module of a Cohen-Macaulay local ring $R$ \cite{EF25}.

Of key significance for our work is the following which can be seen from combining \cite[Proposition 4.1]{DD23} and \cite[Proposition 2.5]{CD25}: 

\begin{prop}\label{introthm1}
Suppose $R$ is a local ring and that $I$ is an ideal of positive grade in $R$. Then $\tr_R(I) \subseteq  \tr_R(\Omega^1_R(I))$.
\end{prop}

It is natural to ask when this behavior extends to modules of higher rank. Even for numerical semigroup rings, this turns out to fail outside exceedingly stringent hypotheses, which we prove with the following:

\begin{theorem}\label{introthm2}
Suppose $R$ is a numerical semigroup ring. Then the following are equivalent:
\begin{enumerate}
\item[$(1)$] We have $\tr_R(M) \subseteq \tr_R(\Omega^1_R(M))$ for all finitely generated torsion-free $R$-modules $M$ having no nonzero free summand.
\item[$(2)$] We have $\tr_R(M) \subseteq \tr_R(\Omega^1_R(M))$ for all finitely-generated torsion-free $R$-modules $M$ having no nonzero free summand and for which $\rank(M)=\mu_R(M)-1$.
\item[$(3)$] Either $e(R) \le 2$ or $\mathfrak{c}_R=\m$, where $e(R)$ denotes the Hilbert-Samuel multiplicity of $R$ and $\mathfrak{c}_R$ is the conductor ideal of $R$. 
\end{enumerate}
\end{theorem}

Theorem \ref{introthm1} implies $\mathfrak{c}_R \subseteq \tr_R(\Omega^1_R(\mathfrak{c}_R))$ whenever $R$ is analytically unramified (this is well-known if $\tr_R(\Omega^1_R(\mathfrak{c}_R))$ admits a principal reduction; see \Cref{traceprop} (9)). Viewing $\tr_R(\Omega^1_R(\mathfrak{c}_R))$ as an invariant of $R$, it is thus natural to wonder, at the opposite end of the spectrum from \Cref{introthm2}, when we have that $\tr_R(\Omega^1_R(\mathfrak{c}_R))$ is as small as possible, i.e., when $\tr_R(\Omega^1_R(\mathfrak{c}_R))=\mathfrak{c}_R$. For numerical semigroup rings, we prove this only occurs under similarly restrictive hypotheses as \Cref{introthm1}.  

\begin{theorem}\label{introthm3}
Suppose $R$ is a numerical semigroup ring and let $A$ be a minimal homogenous presentation matrix for the conductor ideal $\mathfrak{c}_R$ of $R$. Then $\tr_R(\Omega^1_R(\mathfrak{c}_R))=I_1(A)$, and the following are equivalent:
\begin{enumerate}
\item[$(1)$] $\tr_R(\Omega^1_R(\mathfrak{c}_R))=\mathfrak{c}_R$.
\item[$(2)$] $R$ is not a DVR and for some integer $s$, we have $R=k[t^{a_1},t^{sa_1+1},t^{sa_1+2},\dots,t^{sa_1+a_1-1}]$.
\end{enumerate}

\end{theorem}

This work is structured as follows: Section  \ref{prelimsection} contains the necessary background information on trace ideals, Ulrich modules, and reductions needed for the remainder of the paper. Section \ref{mainresultssection} contains our main results on when taking the syzygy induces containment on trace ideals, and in particular proves Theorem \ref{introthm2} (see Theorem \ref{mainconductortm}). We also prove Theorem \ref{tracesyzylrich} which describes, for numerical semigroup rings, the structure of $\tr_R(\Omega^1_R(M))$ when $M$ is a so-called Ulrich module (see \Cref{ulrichdefn} for the definition), and serves a key point in the proofs of both Theorem \ref{introthm2} and \ref{introthm3}. Section  \ref{conductortrace} focuses more specifically on the trace ideals of syzygies of the conductor ideal, calculating an explicit formula for $\tr_R(\Omega^1_R(\mathfrak{c}_R))$ in Theorem \ref{conductorsyzthm} and culminating in the proof of Theorem \ref{introthm3} (see Theorem \ref{syzofconductorthm}).

\section{Preliminaries}\label{prelimsection}

In this section we provide some background needed for our main results. Throughout, we let $(R,\m,k)$ denote either a local ring or a positively graded $k$-algebra over the field $k$ with homogeneous maximal ideal $\m$. Set $d:=\dim(R)$. Unless otherwise stated, all modules are assumed to be finitely generated and in the graded setting to be graded; we write $|x|$ for the degree of a homogeneous element $x \in M$. We let $e_R(M):=\displaystyle \lim_{n \to \infty} \dfrac{d!l_R(M/\m^n M)}{n^{d}}$ denote the Hilbert-Samuel multiplicity of $M$. We let $\mu_R(M):=\dim_k(M/\m M)$ denote the minimal number of generators of $M$. We let $\mathfrak{c}_R$ denote the \emph{conductor ideal} for $R$, so $\mathfrak{c}_R:=(R:_{Q(R)} \bar{R})$ where $\bar{R}$ denotes the integral closure of $R$ and $Q(R)$ is the total quotient ring of $R$.

We write $(-)^*:=\Hom_R(-,R)$. If $A$ is a minimal presentation matrix for $M$, we let $\Tr_R(M):=\coker(A^T)$, known as the \emph{Auslander Transpose} of $M$. The Auslander transpose is a key aspect of the stable module theory of Auslander-Buchweitz \cite{AB69}. We refer to \cite[Section 2]{DT15} as a reference for several of its key properties.

We recall that if $M$ is an $R$-module, then the \emph{trace ideal} of $M$ in $R$ is the image of the natural map $M^* \otimes_R M \to R$ given by $f \otimes x \mapsto f(x)$. We recall the key facts about trace theory that are relevant to our work below. 
\begin{prop}\label{traceprop}
Let $M$ and $N$ be $R$-modules and let $I$ be an ideal in $R$. Then:
\begin{enumerate}
\item[$(1)$] We have $I \subseteq \tr_RI)$ with equality if and only if $I$ is a trace ideal, that is, if we have $I=\tr_R(M)$ for some $M$. In particular, $\tr_R(\tr_R(M))=\tr_R(M)$ for any $R$-module $M$.
\item[$(2)$] If $M$ generates $N$, that is, if there is a surjection $M^{\oplus n} \twoheadrightarrow N$ for some $n$, then we have $\tr_R(N) \subseteq \tr_R(M)$. 
\item[$(3)$] We have $\tr_R(M \oplus N)=\tr_R(M)+\tr_R(N)$.
\item[$(4)$] We have $\tr_R(M) \subseteq \tr_R(M^*)$ with equality if $M$ is reflexive.
\item[$(5)$] We have $\tr_R(M)=R$ if and only if $R$ is a direct summand of $M$.
\item[$(6)$] If $R^{\oplus m} \xrightarrow{A} R^{\oplus n} \xrightarrow{p} M \rightarrow 0$ is a minimal presentation matrix for $M$, then $\tr_R(M)=I_1(B)$ where $B$ is a matrix whose columns form a minimal generating set for $\ker(A^T)$. 
\item[$(7)$] If $I$ contains a nonzerodivisor $x$, then $\tr_R(I)=(R:_{Q(R)} I)I=((x:I)I:x)$.
\item[$(8)$] If $x_1,\dots,x_n$ is a minimal generating set for $I$ where each $x_i$ a nonzerodivisor (such a generating set exists if and only if $I$ has positive grade), then $\tr_R(I)=\sum^n_{i=1} (x_i:I)$.
\item[$(9)$] If $R$ is analytically unramified, and $I$ is a trace ideal of positive grade with a principal reduction, then $\mathfrak{c}_R \subseteq I$. %check hypotheses
\item[$(10)$] If $R$ is analytically unramified, then the conductor $\mathfrak{c}_R$ is a trace ideal.
\end{enumerate}
\end{prop}

\begin{proof}
Items (1)-(5) follow from \cite[Proposition 2.8]{lindo2017trace}, (6) can be found in \cite[Remark 2.5]{lindo2017trace}, and (7) follows from \cite[Proposition 2.4]{kobayashi2019rings}. For item (8), see \cite[Proposition 3.4 (3)]{lyle2024annihilators}. Finally, (9) and (10) follow from \cite[Corollary 3.6]{DS23} and \cite[Lemma 2.6 (1)]{GI20} respectively.
\end{proof}

We recall if $R$ is Cohen-Macaulay (CM) then we always have the so-called Abhyankar bound $e(R) \ge \mu_R(\m)-d+1$. We say $R$ has \emph{minimal multiplicity} if equality is achieved. A notable feature of Cohen-Macaulay rings with minimal multiplicity is their abundance of so-called Ulrich modules, whose definition we recall below:
\begin{defn}\label{ulrichdefn}
An $R$-module $M$ is said to be Ulrich if it is maximal Cohen-Macaulay (MCM) and $e_R(M)=\mu_R(M)$. 
\end{defn}
When $|k|=\infty$, an MCM $R$-module $M$ is Ulrich if and only if $\underline{x}M=\m M$ for some minimal reduction $\underline{x}$ of $\m$. For our purposes, the following serves as the primary examples of Ulrich modules:
\begin{prop}\label{ulrichex}\
\begin{enumerate}
\item[$(1)$] If $R$ is CM with minimal multiplicity then $\Omega^1_R(M)$ is Ulrich if $M$ is MCM.
\item[$(2)$] If $R$ is CM of dimension $1$ with minimal multiplicity, then $M^*$ is Ulrich provided $M$ has no nonzero free summand. In particular, under these hypotheses, if $I$ is an ideal in $R$ containing a nonzerodivisor $x$, then $(x:I)$ is Ulrich.
\item[$(3)$] If $R$ is analytically unramified of dimension $1$, then $\mathfrak{c}_R$ and $\bar{R}$ are Ulrich.
\item[$(4)$] If $R$ is CM of dimension $1$, then $\m^n$ is Ulrich for any $n \gg 0$. In particular, $n$ can be chosen as any value at least the reduction number of $R$ (upon extension of the residue field, if necessary).

\item[$(5)$] If $R$ is reduced of dimension $1$ and $M$ is Ulrich, then $M^*$ is Ulrich.
\end{enumerate}
\end{prop}

\begin{proof}
Item (1) is the content of \cite[Theorem B (1)]{KT19}, while Item (2) can be found in \cite[Lemma 3.12 (2)-(3)]{lyle2024annihilators}. For items (3) and (4), see \cite[Example 2.4 (iii)]{CL23}. Finally, (5) follows from \cite[Theorem A (2)]{KT19}.
\end{proof}

For later purposes, we add to this list of examples via the following:

\begin{prop}\label{traceofulrichisulrich}
Suppose $R$ is CM of dimension $1$. If $M$ is an Ulrich module, then $\tr_R(M)$ is an Ulrich module. 
\end{prop}

\begin{proof}
Extending the residue field if needed, we may suppose that $|k|=\infty$. Then we may choose $y$ to be a minimal reduction for $\m$, and it suffices to show that $\m \tr_R(M) \subseteq y\tr_R(M)$. Take $a \in \m$ and let $f(x)$ be a generator for $\tr_R(M)$, so $f \in M^*$ and $x \in M$. As $M$ is Ulrich, we have $ax=yx'$ for some $x' \in M$. Then $af(x)=f(ax)=f(yx')=yf(x') \in y\tr_R(M)$, so $\tr_R(M)$ is Ulrich. 
\end{proof}

We say $R$ is a numerical semigroup ring if it is ambiguously a ring of the form $R=k[\![t^{a_1},\dots,t^{a_n}]\!]$ or $R=k[t^{a_1},\dots,t^{a_n}]$. We will work in practice with the graded variant $R=k[t^{a_1},\dots,t^{a_n}]$, and note that the properties we consider can be established suitably for $R=k[\![t^{a_1},\dots,t^{a_n}]\!]$ via completion. For a numerical semigroup ring $R$, it is well known that $e(R)=a_1$ and that $\mathfrak{c}_R=(t^{F+1},t^{F+2},\dots,t^{F+a_1})$ where $F$ is the \emph{Frobenius number} of $R$, that is, $F=\sup\{i \mid t^i \notin R\}$. We point to \cite{RG09} as a reference for these facts and other fundamental properties of numerical semigroup rings.

We will need the following well-known fact about conductor ideals. We include a short proof due to lack of a reference covering the precise statement we need.
\begin{prop}\label{conductorfact}
Suppose $R$ is analytically unramified. If $y$ is a minimal reduction of $\mathfrak{c}_R$, then $(y:\mathfrak{c}_R)=\mathfrak{c}_R$.

\end{prop}

\begin{proof}
We have $(R:_{Q(R)} \mathfrak{c})=(R:_{Q(R)} (R:_{Q(R)} \bar{R}))=\bar{R}$. Then $y(R:_{Q(R)} \mathfrak{c}_R)=y\bar{R}=\mathfrak{c}_R\bar{R}=\mathfrak{c}_R$. The claim follows as $y(R:_{Q(R)} \mathfrak{c}_R)=(y:\mathfrak{c}_R)$.
\end{proof}

\section{{When taking syzygies induces containment of traces: $\tr_R(M) \subseteq \tr_R(\Omega^1_R(M))$}}\label{mainresultssection}
This section contains our main results regarding the behavior of trace ideals under taking syzygies. We begin with the following example indicating that there need not be any direct containment between $\tr_R(M)$ and $\tr_R(\Omega^1_R(M))$ in general:
\begin{example}\label{incomparable}
Let $R=k[\![x,y]\!]/(xy)$, and let $I=(x)$. Then $\tr_R(I)=I$ and $\tr_R(\Omega^1_R(I))=(y)$. In particular, $\tr_R(I)$ and $\tr_R(\Omega^1_R(I))$ are incomparable.
\end{example}

\begin{proof}
If $f:(x) \to R$, and if $y \in (x)$, then $yf(x)=f(yx)=f(0)=0$. So $f(x) \in (0:y)=(x)$. It follows that $I$ is a trace ideal. We have $\Omega^1_R((x)) \cong \Omega^1_R(R/(y)) \cong (y)$. By symmetry, $(y)$ is also a trace ideal, so $\tr_R(\Omega^1_R((x)))=(y)$, as claimed.
\end{proof}

We next note the following proposition that, while elementary, highlights a key point that we will return to frequently.
\begin{prop}\label{I1colon}
Suppose $R$ is a local ring and let $M$ be an $R$-module with minimal generating set $x_1,\dots,x_n$. If $A$ is a minimal presentation matrix for $M$, then 
\[I_1(A)=\sum_{i=1}^n ((x_1,\dots,\hat{x_i},\dots,x_n): x_i).\]
\end{prop}

\begin{proof}
If $\sum^n_{i=1} r_ie_i \in \Omega^1_R(M)$ then $\sum^n_{i=1} r_ix_i=0$ so $r_i \in ((x_1,\dots,\hat{x_i},\dots,x_n):x_i)$ for each $i$. In particular, $I_1(A) \subseteq ((x_1,\dots,\hat{x_i},\dots,x_n): x_i)$. On the other hand, if $s_i \in ((x_1,\dots,\hat{x_i},\dots,x_n): x_i)$ then we may write $s_ix_i=\sum_{j \ne i} s_jx_j$ for some elements $s_j \in R$. So $s_ie_i-\sum_{j \ne i} s_je_j \in \Omega^1_R(M)$ which forces $s_i \in I_1(A)$.
\end{proof}

The next result gives a slight sharpening of \cite[Proposition 2.5]{CD25}:

\begin{prop}\label{idealsyztrace}
Suppose $R$ is a local ring and $M$ a finitely generated $R$-module with minimal generating set $x_1,\dots,x_n$ and with corresponding minimal presentation matrix $A$. Then

\[I_1(A) \subseteq \sum_{i=1}^n \tr_R((x_1,\dots,\hat{x}_i,\dots,x_n):x_i) \subseteq \tr_R(\Omega^1_R(M)).\]
If $\Ext^1_R(M,R)=0$, then these containments are all equalities.

\end{prop}

\begin{proof} 

From \Cref{I1colon}, we have $I_1(A)=\sum^n_{i=1} ((x_1,\dots,\hat{x}_i,\dots,x_n):x_i)$, and so $I_1(A) \subseteq \sum^n_{i=1} \tr_R((x_1,\dots,\hat{x}_i,\dots,x_n:x_i)$ from \Cref{traceprop} (1).

For the second containment, let $p_i:R^{\oplus n} \to R$ denote projection onto the $i$th component, and let $f_i:\Omega^1_R(M) \to R$ be the restriction of $p_i$ to $\Omega^1_R(M)$. Then by definition of $\Omega^1_R(M)$, we have $p_i:\Omega^1_R(M) \to (x_1,\dots,\hat{x}_i,\dots,x_n):x_i)$. Then from \Cref{traceprop} (2), we see that $\tr_R((x_1,\dots,\hat{x}_i,\dots,x_n):x_i)) \subseteq \tr_R(\Omega^1_R(M))$ for all $i$, and so $\sum_{i=1}^n \tr_R((x_1,\dots,\hat{x}_i,\dots,x_n):x_i) \subseteq \tr_R(\Omega^1_R(M))$. 

If $\Ext^1_R(M,R)=0$, then let $B$ be a minimal presentation matrix for $\Omega^1_R(M)$ so that there is an exact sequence
\[R^{\oplus \ell} \xrightarrow{B} R^{\oplus m} \xrightarrow{A} R^{\oplus n} \xrightarrow{p} M \rightarrow 0.\]
Applying $(-)^*$, then since $\Ext^1_R(M,R)=0$, we have an exact sequence 
\[0 \rightarrow M^* \rightarrow R^{\oplus n} \xrightarrow{A^T} R^{\oplus m} \xrightarrow{B^T} R^{\oplus l} \rightarrow \Tr_R(\Omega^1_R(M)) \rightarrow 0,\]
and it follows from \Cref{traceprop} (6) that $I_1(A^T)=I_1(A)=\tr_R(\Omega^1_R(M))$.

\end{proof}

%The following Corollary gives a variation on \cite[Proposition 1.7]{EF25}:

As an immediate consequence we have the following, which in particular recovers results of Dey-Dutta (see \cite[Proposition 4.1 and Lemma 4.3]{DD23}):

\begin{cor}\label{tracesyzidealcor}
Suppose $I$ is an ideal in $R$ of positive grade. Then $\tr_R(I) \subseteq I_1(A) \subseteq  \tr_R(\Omega^1_R(I))$. 
\end{cor}

\begin{proof}
By prime avoidance, we may choose a minimal generating set $x_1,\dots,x_n$ for $I$ where each $x_i$ is a nonzerodivisor. Then from \Cref{traceprop} (8), we have \[\tr_R(I)=\sum^n_{i=1} (x_i:I)=\sum^n_{i=1} \bigcap_{j \ne i} (x_i:x_j) \subseteq \sum^n_{i=1} (x_1,\dots,\hat{x}_i,\dots,x_n:x_i).\]
Combining \Cref{I1colon} and \Cref{idealsyztrace} then gives the claim.
\end{proof}

The following shows that, in the setting of \Cref{idealsyztrace}, we do not have $\tr_R(\Omega^1_R(M))=I_1(A)$ in general.

\begin{example}\label{assumptionsneeded}
Let $R=k[t^5,t^9,t^{11},t^{17}]$ and take $I=(t^9,t^{11})$. Then $I_1(A)=\tr_R(I) \subsetneq \tr_R(\Omega^1_R(I))$.
\end{example}

\begin{proof}
From \Cref{I1colon}, we have $I_1(A)=(t^9:t^{11})+(t^{11}:t^9)$, which equals $\tr_R(I)$ by \Cref{traceprop} (8). We have $I^* \cong (t^9:t^{11})$ by \cite[Lemma 3.3]{HH05}, and we see directly that $(t^9,t^{15},t^{16},t^{17}) \subseteq (t^9:t^{11})$.  
The only $t^a$ in $R$ that are not in $(t^9,t^{15},t^{16},t^{17})$ are $1,t^5,t^{10}$, and $t^{11}$. We check directly that each of these is not in $(t^9:t^{11})$, so $(t^9:t^{11})=(t^9,t^{15},t^{16},t^{17})$. By a similar calculation, we have $(t^{11}:t^{9})=(t^{11},t^{17},t^{18},t^{19})$, and so $\tr_R(I)=(t^9,t^{11},t^{15},t^{17})$. But we observe that $t^{10} \in (t^9:(t^9:I))=(t^9:(t^9,t^{15},t^{16},t^{17}))$, and then by \Cref{traceprop} (8), we have $t^{10} \in \tr_R(I^*)$. By \Cref{traceprop} (4), we have $\tr_R(I) \subsetneq \tr_R(I^*)=\tr_R(\Omega^1_R(I))$, as claimed. 
\end{proof}

\Cref{idealsyztrace} and \Cref{tracesyzidealcor} give lower estimate for the trace ideal of a syzygy. For homogeneous ideals in numerical semigroup rings, we also have a upper estimate:
\begin{prop}\label{containedinsumtrace}
Suppose $R=k[t^{a_1},\dots,t^{a_n}]$ is a numerical semigroup ring, and let $I$ be a homogeneous ideal with minimal homogeneous generating set $x_1,\dots,x_m$ and minimal homogeneous presentation matrix $A$. Then 
\[I_1(A)=\sum_{j=1}^m \sum_{i \ne j} (x_i:x_j) \subseteq \tr_R(\Omega^1_R(I)) \subseteq \sum_{j=2}^m \sum_{i<j} \tr_R(x_i:x_j)\] with equalities if each of the two-generated ideals $(x_i,x_j)$ for $i \ne j$ is reflexive, e.g. if $R$ is Gorenstein.
\end{prop}

\begin{proof}
From \cite[Lemma 3.19]{lyle2024annihilators} we have $((x_1,\dots,\hat{x}_j,\dots,x_n):x_j)=\sum_{i \ne j} (x_i:x_j)$, and so $\sum_{j=1}^m \sum_{i \ne j} (x_i:x_j) \subseteq \tr_R(\Omega^1_R(I))$ from \Cref{idealsyztrace}.

It follows from \cite[Lemma 3.21 (2)]{lyle2024annihilators} that there is a surjection $\bigoplus^m_{j=1} \bigoplus_{i \ne j} (x_i:x_j) \twoheadrightarrow \Omega^1_R(I)$. That $\tr_R(\Omega^1_R(I)) \subseteq \sum_{j=1}^m \sum_{i \ne j} \tr_R(x_i:x_j)$ then follows from \Cref{traceprop} (2) and (3). We note that $(x_i:x_j)=(x_i:(x_i,x_j)) \cong (x_i,x_j)^* \cong (x_j:(x_i,x_j))=(x_j:x_i)$ from \cite[Lemmas 2.4.1-2.4.3]{HS06}, so $\tr_R(x_i:x_j)=\tr_R(x_j:x_i)$. It follows that $\sum_{j=1}^m \sum_{i \ne j} \tr_R(x_i:x_j)=\sum_{j=2}^m \sum_{i<j} \tr_R(x_i:x_j)$.

If each $(x_i,x_j)$ with $i \ne j$ is reflexive, then as above, we have $(x_i:x_j) \cong (x_i,x_j)^*$, and it follows from \Cref{traceprop} (4) and (8) that $\tr_R(x_i:x_j)=\tr_R(x_i,x_j)=(x_i:x_j)+(x_j:x_i)$ for each $i \ne j$. We thus see that $\sum_{j=1}^m \sum_{i \ne j} (x_i:x_j)=\sum_{j=2}^m \sum_{i \ne j} \tr_R(x_i:x_j)$, completing the proof.
\end{proof}

In light of \Cref{idealsyztrace} it is natural to ask when we can expect $\tr_R(M) \subseteq \tr_R(\Omega^1_R(M))$ for any $R$-module $M$, and moreover when we will have $\tr_R(M)=\tr_R(\Omega^1_R(M))$. A key obstruction to the first is the fact that $\Omega^1_R(M)=\Omega^1_R(M \oplus R)$, so we restrict our consideration to modules without free summands. Even still, these conditions turn out to be exceedingly strong. The remainder of this section is devoted to characterizing this behavior for numerical semigroup rings. In this endeavor we need the following, of independent interest, which we will use to describe the form of $\Omega^1_R(M)$ when $M$ is Ulrich:

\begin{theorem}\label{burchthm}
Suppose $R=k[t^{a_1},\dots,t^{a_n}]$ is a numerical semigroup ring and let $M$ be a homogeneous torsion-free $R$-module with minimal homogeneous generating set $x_1,\dots,x_m$. If for some $i$, we have $\m x_i \subseteq t^{a_1}M$, then $((x_1,\dots,\hat{x}_i,\dots,x_n):x_i)=(t^{d_ia_1},t^{a_2},\dots,t^{a_n})$ for some $d_i$. In particular, $\tr_R(\Omega^1_R(M))=(t^{va_1},t^{a_2},\dots,t^{a_n})$ for some $v \le d_i$.
\end{theorem}

\begin{proof}
Set $L_i=((x_1,\dots,\hat{x_i},\dots,x_m):x_i)$. For any $j \ge 2$, we have $t^{a_j}x_i \in t^{a_1}M$. It follows that there are homogeneous elements $c_1,\dots,c_m \in (t^{a_1})$ with $t^{a_j}x_i=\sum^m_{\ell=1} c_{\ell}x_{\ell}$, and we suppose moreover that each $c_{\ell}$ is either $0$ or has $|c_{\ell}|+|x_{\ell}|=a_j+|x_i|$. In particular, $|c_i|=a_j$. We have $(t^{a_j}-c_i)x_i=\sum_{\ell \ne i} c_{\ell}x_{\ell}$. Since $|c_i|=a_j$, $t^{a_j}-c_i=rt^{a_j}$ for some $r \in k$. From above, we have $rt^{a_j} \in L_i$. If $r=0$, then we would have $t^{a_j}=c_i \in (t^{a_1})$, which is not the case, so $r$ is a unit, and it follows that $t^{a_j} \in L_i$. As this holds for all $j$, we have $(t^{a_2},\dots,t^{a_n}) \subseteq L_i$. Since $L$ is $\m$-primary, there is a $d_i$ for which $t^{d_ia_1} \in L_i$ but where $t^{(d_i-1)a_1} \notin L_i$. But $R/L_i$ is a principal ideal ring whose ideals are uniquely determined by the smallest power of $t^{a_1}$ they contain, and it follows that $L_i=(t^{d_ia_1},t^{a_2},\dots,t^{a_n})$. From  \Cref{idealsyztrace}, we have $L_i \subseteq \tr_R(\Omega^1_R(M))$, so again as $R/L_i$ is a principal ideal ring, $\tr_R(\Omega^1_R(M))=(t^{va_1},\dots,t^{a_n})$ for some integer $1 \le v \le d_i$, as claimed.
\end{proof}

\begin{remark}
By \cite[Lemma 3.9]{DK23}, the condition that $\m x_i \subseteq t^{a_1} x_i$ for some $i$ in the setting of \Cref{burchthm} says that $M$ is a \emph{Burch submodule} of some $R$-module $X$ in the sense of \cite[Definition 3.1]{DK23}.
\end{remark}

A consequence of \Cref{burchthm} is the following:

\begin{theorem}\label{tracesyzylrich}
Suppose $R:=k[t^{a_1},\ldots,t^{a_n}]$ is a numerical semigroup ring. If $M$ is a homogeneous Ulrich module with minimal homogeneous generating set $x_1,\dots,x_m$, then for each $i$, $((x_1,\dots,\hat{x_i},\dots,x_m):x_i)=(t^{d_ia_1},t^{a_2},\dots,t^{a_n})$ for some integer $d_i$. Moreover, if $A$ is a minimal homogeneous presentation matrix for $M$, then $I_1(A)=(t^{\min\{d_1,\ldots,d_m\}},t^{a_2},\dots,t^{a_n})$. In particular, $\tr_R(\Omega^1_R(M))=(t^{da_1},\dots,t^{a_n})$ for some integer $1 \le d \le \min\{d_1,\dots,d_m\}$.
\end{theorem}

\begin{proof}
Since $M$ is Ulrich, we have $\m x_i \subseteq t^{a_1} M$ for each $i$, and the claim follows at once from combining \Cref{I1colon} and \Cref{burchthm}.
\end{proof}

\begin{example}
Let $R=k[t^4,t^{10},t^{15}]$. The semigroup $\langle 4,10,15 \rangle$ is symmetric, so $R$ is Gorenstein, and has Frobenius number $F=21$, so $\mathfrak{c}_R=(t^{22},t^{23},t^{24},t^{25})$. As $R$ is Gorenstein, it follows from \Cref{containedinsumtrace} that 
\[\tr_R(\Omega^1_R(\mathfrak{c}_R))=((t^{23},t^{24},t^{25}):t^{22})+((t^{22},t^{24},t^{25}):t^{23})+((t^{22},t^{23},t^{25}):t^{24})+((t^{22},t^{23},t^{24}):t^{25}).\]
We calculate directly that $((t^{23},t^{24},t^{25}):t^{22})=((t^{22},t^{24},t^{25}):t^{23})=(t^{10},t^{12},t^{15})$, and that $((t^{22},t^{23},t^{25}):t^{24})=((t^{22},t^{23},t^{24}):t^{25})=(t^8,t^{10},t^{15})$. So $\tr_R(\Omega^1_R(\mathfrak{c}_R))=(t^8,t^{10},t^{15})$.
\end{example}

\begin{remark}
We note that $t^{sa_1}$ may not be a minimal generator of $(t^{sa_1},t^{a_2},\dots,t^{a_n})$, as it may be that $t^{sa_1}$ already belongs to $(t^{a_2},\dots,t^{a_n})$. 
\end{remark}

\begin{prop}\label{nofreesummand}
Suppose $R$ is a domain, let $I$ be a nonprincipal ideal in $R$, and let $A$ be a minimal presentation matrix for $I$. Then $\tr_R(\im(A^T)) \ne R$. 
\end{prop}

\begin{proof}
Let $x_1,\dots,x_n$ be a minimal generating set for $I$ with $p:R^{\oplus n} \to I$ given by $p(e_i)=x_i$ and fitting into an exact sequence $R^{\oplus m} \xrightarrow{A} R^{\oplus n} \xrightarrow{p} I \rightarrow 0$. Let $f_1,\dots,f_s$ be a minimal generating set for $I^*$. 

There is a commutative diagram with exact rows:
% https://q.uiver.app/#q=WzAsOCxbMCwwLCIwIl0sWzEsMCwiSV4qIl0sWzIsMCwiKFJee1xcb3BsdXMgbn0pXioiXSxbMywwLCIoUl57XFxvcGx1cyBtfSleKiJdLFswLDEsIjAiXSxbMSwxLCJJXioiXSxbMiwxLCJSXntcXG9wbHVzIG59Il0sWzMsMSwiUl57XFxvcGx1cyBtfSJdLFsxLDUsIiIsMCx7InN0eWxlIjp7InRhaWwiOnsibmFtZSI6Imhvb2siLCJzaWRlIjoidG9wIn19fV0sWzIsNiwiXFxjb25nIl0sWzMsNywiXFxjb25nIl0sWzIsMywiQV4qIl0sWzYsNywiQV5UIl0sWzEsMiwicF4qIl0sWzUsNiwiQiJdLFswLDFdLFs0LDVdXQ==
\[\begin{tikzcd}[cramped]
	0 & {I^*} & {(R^{\oplus n})^*} & {(R^{\oplus m})^*} \\
	0 & {I^*} & {R^{\oplus n}} & {R^{\oplus m}}
	\arrow[from=1-1, to=1-2]
	\arrow["{p^*}", from=1-2, to=1-3]
	\arrow[equals, from=1-2, to=2-2]
	\arrow["{A^*}", from=1-3, to=1-4]
	\arrow["\alpha", from=1-3, to=2-3]
	\arrow["\beta", from=1-4, to=2-4]
	\arrow[from=2-1, to=2-2]
	\arrow["\theta", from=2-2, to=2-3]
	\arrow["{A^T}", from=2-3, to=2-4]
\end{tikzcd}\]
where $\alpha$ and $\beta$ are the isomorphisms given by $\alpha(f)= \sum^n_{i=1} f(e_i)e_i$ and $\beta(g)=\sum^m_{i=1} g(e_i)e_i$, and where $\theta$ is given by $\theta(\phi)=\sum^n_{i=1} \phi(x_i)e_i$. In particular, \[B=\begin{pmatrix} f_1(x_1) & f_2(x_1) & \cdots & f_s(x_1) \\ f_1(x_2) & f_2(x_2) & \cdots & f_s(x_2) \\ \vdots & \vdots & \ddots & \vdots \\ f_1(x_n) & f_2(x_n) & \cdots & f_s(x_n) \end{pmatrix}\]
is a presentation matrix for $\im(A^T)$. Then $R$ is a summand of $\im(A^T)$ if and only if there is a row $\begin{pmatrix} f_1(x_i) & f_2(x_i) & \cdots & f_s(x_i) \end{pmatrix}$ of $B$ which is a linear combination of the others. This in particular means that for some $i$ and $r_1,\dots,\hat{r_i},\dots, r_n \in R$, we have $f_1(x_i)=\sum_{\ell \ne i} r_{\ell}f_1(x_{\ell})=f_1(\sum_{\ell \ne i} r_{\ell}x_{\ell})$. But as $R$ is a domain, every nonzero map in $I^*$ is injective, and it follows that $x_i=\sum_{\ell \ne i} r_{\ell}x_{\ell}$, contradicting the assumption that $x_1,\dots,x_n$ is a minimal generating set for $I$. Therefore $\im(A^T)$ cannot have $R$ as a summand, and so $\tr_R(\im(A^T)) \ne R$ by \Cref{traceprop} (5).
\end{proof}

\begin{prop}\label{mpowern}
Suppose $R:=k[t^{a_1},\dots,t^{a_n}]$ is a numerical semigroup ring. Set $y:=t^{a_1}$, let $\ell$ be any integer for which $y\m^{\ell}=\m^{\ell+1}$, and let $A$ be a minimal homogenous presentation matrix for $\m^{\ell}$. Then we have the following:
\begin{enumerate}
\item[$(1)$] $\tr_R(\m^{\ell})=(y^{\ell}:\m^{\ell})$.
\item[$(2)$] $I_1(A)=\m$.
\end{enumerate}
\end{prop}

\begin{proof}
For $(1)$, from \Cref{traceprop} (7) we have $y^{\ell}\tr_R(\m^{\ell})=(y^{\ell}:\m^{\ell})\m^{\ell}$. But from \Cref{ulrichex} (4), we have that $(y^\ell:\m^\ell)$ is Ulrich, so $(y^{\ell}:\m^{\ell})\m^{\ell}=(y^{\ell}:\m^{\ell})y^{\ell}$. Then $y^{\ell}\tr_R(\m^{\ell})=y^{\ell}(y^{\ell}:\m^{\ell})$ which forces $\tr_R(\m^{\ell})=(y^{\ell}:\m^{\ell})$. 

For $(2)$, it follows from \Cref{tracesyzylrich} that $(t^{a_2},\dots,t^{a_n}) \subseteq I_1(A)$, so it suffices to show that $y \in I_1(A)$. We claim that $y^{\ell},t^{a_2}y^{\ell-1}$ is part of a minimal generating set for $\m^\ell$. Indeed, suppose for some homogeneous $a,b \in R$ that $ay^{\ell}+bt^{a_2}y^{\ell-1} \in \m^{\ell+1}=y\m^\ell$. Then $ay^{\ell-1}+bt^{a_2}y^{\ell-2} \in \m^{\ell}$. The minimal degree of an element in $\m^{\ell}$ is $\ell a_1$, and it follows that $a \in \m$. This forces $bt^{a_2}y^{{\ell}-2} \in \m^{\ell}$. If $b$ is not in $\m$, then as it is homogeneous, it must be a unit, which would force $t^{a_2}y^{{\ell}-2} \in \m^{\ell}$. So $t^{a_2}y^{{\ell}-2}=t^{c_1}\cdots t^{c_{\ell}}$. As each $c_i \ge a_1$, we either have $c_i=a_1$ for all $i$, which would force $t^{a_2} \in (y)$, which is not the case, or else $|t^{c_1} \cdots t^{c_{\ell}}|>|t^{a_2}y^{\ell-1}|$. It follows that $b$ must be in $\m$, so $y^{\ell},t^{a_2}y^{{\ell}-1}$ form part of a minimal homogeneous generating set for $\m^{\ell}$. But then as $y(t^{a_2}y^{\ell-1})=t^{a_2}y^{\ell}$, it follows from Proposition \ref{I1colon} that $y \in I_1(A)$, so $I_1(A)=\m$. 
\end{proof}

\begin{lem}\label{Rhasminmult}
Suppose $R:=k[t^{a_1},\dots,t^{a_n}]$ is a numerical semigroup ring. If for all torsion-free $R$-modules $M$ with $\rank(M)=\mu_R(M)-1$ and for which $\tr_R(M)=\m$, we have $\tr_R(\Omega^1_R(M))=\m$, then $R$ has minimal multiplicity.
\end{lem}

\begin{proof}
We may suppose that $R$ is not a DVR. Let $\ell$ be the reduction number for $R$ and let $A$ be a minimal homogeneous presentation matrix for $\m^{\ell}$. Then there is a short exact sequence $0 \to (\m^{\ell})^* \to R^{\oplus \mu_R(\m^{\ell})} \to \im(A^T) \to 0$. We claim that $\mu_R(\im(A^T))=\mu_R(\m^{\ell})$. Indeed, the exact sequence above immediately gives that $\mu_R(\im(A^T)) \le \mu_R(\m^{\ell})$. If $\mu_R(\im(A^T))<\mu_R(\m^{\ell})$, then this would force $(\m^{\ell})^*$ to have $R$ as a summand, which would force $(\m^{\ell})^* \cong R$ by rank considerations. But this would force $\m^{\ell}$ to be principal by \cite[Lemma 4.1]{lindo2017trace} (cf. \cite[Lemma 3.9]{DE21}). As $\m^{\ell}$ is Ulrich, this would mean that $R$ is a DVR, contradicting our assumption to the contrary.

From Proposition \ref{mpowern} (2) and \Cref{nofreesummand}, we have $\m=I_1(A)=I_1(A^T) \subseteq \tr_R(\im(A^T)) \ne R$ so that $\tr_R(\im(A^T))=\m$. By additivity of rank, we have $\rank(\im(A^T))=\mu_R(\m^\ell)-1=\mu_R(\im(A^T))-1$. Then by assumption we have $\tr_R((\m^{\ell})^*)=\m$. But $(\m^{\ell})^* \cong (y^{\ell}:\m^{\ell})$ by \cite[Lemmas 2.4.1-2.4.3]{HS06}, so from Proposition \ref{mpowern} $(1)$, we have $\tr_R((\m^{\ell})^*)=\tr_R(y^{\ell}:\m^{\ell})=(y^{\ell}:\m^{\ell})$. It follows that $(y^{\ell}:\m^{\ell})=\m$. But $(y^{\ell}:\m^{\ell})$ is Ulrich, from \Cref{ulrichex} (5), and then so is $\m$. This means $y\m=\m^2$, so that $R$ has minimal multiplicity.
\end{proof}

As the following example demonstrates, even minimal multiplicity is insufficient to force the condition $\tr_R(M) \subseteq \tr_R(\Omega^1_R(M))$ to hold: 

\begin{example}\label{minmultexample}
Let $R=k[t^3,t^5,t^7]$ and set $M:=\Omega^1_R(\Tr_R(\mathfrak{c}_R))$. Then $\tr_R(M)=\m$ while $\tr_R(\Omega^1_R(M))=\mathfrak{c}_R$. In particular, $\tr_R(M) \nsubseteq \tr_R(\Omega^1_R(M))$.
\end{example}

\begin{proof}
The Frobenius number of $R$ is $4$, so $\mathfrak{c}_R=(t^5,t^6,t^7)$. We claim that 
\[A=\begin{pmatrix} t^5 & t^6 & t^7 & 0 & 0 & 0 \\ 0 & -t^5 & 0 & t^6 & t^7 & t^8 \\ -t^3 & 0 & -t^5 & -t^5 & -t^6 & -t^7  \end{pmatrix}\] is a minimal homogeneous presentation matrix for $\mathfrak{c}_R$. Let $p:R^{\oplus 3} \to M$ be given by $e_1 \mapsto t^5$, $e_2 \mapsto t^6$, and $e_3 \mapsto t^7$. We observe directly that $\im(A) \subseteq \ker(p)=\Omega^1_R(\mathfrak{c}_R)$. Then there is a short exact sequence
\[0 \rightarrow K \rightarrow C \xrightarrow{q} \mathfrak{c}_R \rightarrow 0\]
where $C:=\coker(A)$, where $q$ is the natural map, and where $K:=\ker(q)$. Applying $- \otimes_R R/t^3 R$, we have, since $\mathfrak{c}_R$ is torsion-free, an exact sequence
\[0 \rightarrow K \otimes_R R/t^3 R \rightarrow C \otimes_R R/t^3 R \xrightarrow{q \otimes R/t^3 R} \mathfrak{c}_R \otimes_R R/t^3 R \rightarrow 0.\]

We also have an exact sequence $(R/t^3 R)^{\oplus 6} \xrightarrow{\bar{A}} (R/t^3R)^{\oplus 3} \rightarrow C \otimes_R R/t^3 R \rightarrow 0$. Applying elementary column operations, we see that 
\[\im(\bar{A})=\begin{pmatrix} t^5 & 0 & t^7 & 0 & 0 & 0 \\ 0 & -t^5 & 0 & 0 & t^7 & 0 \\ 0 & 0 & -t^5 & -t^5 & 0 & -t^7  \end{pmatrix} \cong \begin{pmatrix} t^5 & 0 & 0 & t^7 & 0 & 0 \\ 0 & t^5 & 0 & 0 & t^7 & 0 \\ 0 & 0 & t^5 & 0 & 0 & t^7  \end{pmatrix}.\]
Noting that $R/t^3 R \cong k[x,y]/(x^2,xy,y^2)$, each column of this matrix is a distinct socle generator, and it follows that $\im(\bar{A}) \cong k^{\oplus 6}$. As $l_R(R/t^3 R)=e(R)=3$, it follows from additivity of length that $l_R(C \otimes_R R/t^3 R)=9-6=3$. But $\mathfrak{c}_R \otimes R/t^3 R \cong k^{\oplus 3}$, and thus $l_R(K \otimes_R R/t^3R)=0$. But this forces $K=0$, so that $q$ is an isomorphism. But as $q$ is the natural map, this means $\im(A)=\Omega^1_R(\mathfrak{c}_R)$.

Thus 
\[M=\im(A^T)=\begin{pmatrix} t^5 & 0 & -t^3 \\ t^6 & -t^5 & 0 \\ t^7 & 0 & -t^5 \\ 0 & t^6 & -t^5 \\ 0 & t^7 & -t^6 \\ 0 & t^8 & -t^7 \end{pmatrix}.\]
Let $x_i$ denote the ith column of $A^T$. Then $\im(A^T)$ will have a nonzero free summand if and only if row operations on a corresponding presentation for $\im(A^T)$ produce a row of $0$'s. But $t^6x_1+t^7x_2+t^8x_3=0$. For $t^6,t^7,t^8$, none of these is in the ideal generated by the other two, and so $\im(A^T)$ cannot have $R$ as a summand. On the other hand, restricting the projection maps $R^{\oplus 6} \to R$ to $M$, we see that $\m \subseteq \tr_R(M)$ and it follows that $\tr_R(M)=\m$. But $\Omega^1_R(M) \cong \Omega^2_R(\Tr_R(\mathfrak{c}_R)) \cong \mathfrak{c}_R^* \cong \mathfrak{c}_R$. So $\tr_R(\Omega^1_R(M))=\mathfrak{c}_R$, completing the proof. 
\end{proof}

In order to provide our characterization for the $\tr_R(M) \subseteq \tr_R(\Omega^1_R(M))$ behavior, we must explore properties of numerical semigroup rings of minimal multiplicity a bit further than what is known in the literature. In particular, we prove the following unique to the minimal multiplicity setting:

\begin{prop}\label{tracemequiv}
Suppose $R$ is Cohen-Macaulay of dimension $1$ and has minimal multiplicity. If $M$ is a torsion-free $R$-module, then the following are equivalent:
\begin{enumerate}
\item[$(1)$] $\tr_R(M)=\m$.
\item[$(2)$] $\tr_R(M^*)=\m$.
\end{enumerate}
\end{prop}

\begin{proof}
If $\tr_R(M)=\m$, it follows from \Cref{traceprop} (4) that $\m \subseteq \tr_R(M^*)$. On the other hand, if $\tr_R(M^*)=R$, then $R$ is a summand of $M^*$, and then $R$ would be a summand of $M$ by \cite[Corollary 3.10]{lyle2024annihilators}. As this is not the case, we must have $\tr_R(M^*)=\m$. 

Now suppose $\tr_R(M^*)=\m$. If $R$ is a summand of $M$, then it would also be a summand of $M^*$, so we have $\tr_R(M) \ne R$. From \cite[Lemma 3.12 (6)]{lyle2024annihilators}, we have that $\tr_R(M)$ is Ulrich as an $R$-module. So $t^{a_1}\tr_R(M)=\m \tr_R(M)$. In particular, $\tr_R(M) \cong \m \tr_R(M)$, and $\tr_R(M)$ inherits a module structure over $\End_R(\m)$ from that of $\m \tr_R(M)$. From \cite[Corollary 3.24]{lindo2017trace} we have $\End_R(\tr_R(M)) \cong \End_R(\tr_R(M^*))=\End_R(\m)$. Since $R$ has minimal multiplicity we have an isomorphism of $\End_R(\m)$-modules, $\End_R(\m) \cong \m$; see \cite[3.12]{CC25}. Noting that $\Hom_R(A,B)=\Hom_{\End_R(\m)}(A,B)$ whenever $A,B$ are $\End_R(\m)$-modules with $B$ torsion-free (see e.g. \cite[Lemma 6.3]{lindo2017trace}), we have isomorphisms of $\End_R(\m)$-modules
\[\End_R(\tr_R(M)) \cong \Hom_R(\tr_R(M),\m) \cong \Hom_{\End_R(\m)}(\tr_R(M),\m) \cong \Hom_{\End_R(\m)}(\tr_R(M),\End_R(\m)).\]
So $\Hom_{\End_R(\m)}(\tr_R(M),\End_R(\m)) \cong \End_R(\m)$ and it follows from \cite[Lemma 3.9]{DE21} that $\tr_R(M) \cong \End_R(\m) \cong \m$, but then $\tr_R(\tr_R(M))=\tr_R(M)=\m$, as claimed.
\end{proof}

\begin{lem}\label{colonidealtrace}
Suppose $R=k[t^{a_1},\dots,t^{a_n}]$ be a numerical semigroup ring of minimal multiplicity and let $I=(t^a,t^b)$ be a homogeneous non-principal ideal in $R$ with $a<b$. Then $\tr_R(t^a:t^b)=\m$ if and only if $(t^a:t^b)=\m$. 
\end{lem}

\begin{proof}
As $(t^a:t^b) \cong I^*$ from \cite[Lemma 3.3]{HH05}, it follows from Proposition \ref{tracemequiv} that $\tr_R(t^a,t^b)=\m$ if and only if $\tr_R(I)=\m$. But since $a<b$, $t^a$ is a minimal reduction of $I$, and it follows from \cite[Theorem 4.8]{lyle2024annihilators} that $\tr_R(I)=\m$ if and only if $(t^a:I)=(t^a:t^b)=\m$.
\end{proof}

The following is well-known, but we include a proof due to lack of a suitable reference:

\begin{lem}\label{gensetexists}
Suppose $R$ is a domain, let $M$ be a finitely generated $R$-module, and let $r:=\rank(M)$. If $x_1,\dots,x_n$ is a minimal generating set for $M$, then there is a subcollection $x_{i_1},\dots,x_{i_r}$ of the $x_i$ for which $\langle x_{i_1},\dots,x_{i_r} \rangle$ is free. 
\end{lem}

\begin{proof}
The images $x_1 \otimes 1,x_2 \otimes 1,\dots,x_n \otimes 1$ in $M \otimes_R Q(R)$ form a generating set for $M \otimes_R Q(R)$. Since $M$ has rank $r$, $M \otimes_R Q(R)$ is a vector space over $Q(R)$ of dimension $r$, and we may thus trim the $x_i \otimes 1$ so that $x_{i_1} \otimes 1,\dots,x_{i_r} \otimes 1$ form a basis for $M \otimes_R Q(R)$. Let $F=\langle x_{i_1},\dots,x_{i_r} \rangle$. Since $Q(R)$ is a flat $R$-module, the natural inclusion $i:F \to M$ induces an injection $i \otimes 1:F \otimes_R Q(R) \to M \otimes_R Q(R)$. In particular, $F$ has rank $r$. But since $F$ is $r$-generated, there is an exact sequence
\[0 \to \Omega^1_R(F) \to R^{\oplus r} \to F \to 0,\]
and applying $- \otimes_R Q(R)$ to this sequence reveals that $\Omega^1_R(F) \otimes_R Q(R)=0$. But $\Omega^1_R(F)$ is torsionless, and so it can only be that $\Omega^1_R(F)=0$, which forces $F$ to be free.
\end{proof}

The following lemma extends the familiar fact that $\Omega^1_R(I) \cong (x_1:I)$ when $I:=(x_1,x_2)$ is a two generated ideal of positive grade (see \cite[Lemma 3.3]{HH05}). 

\begin{lem}\label{rankoneoff}
Suppose $M$ is an $R$-module of constant rank $r:=\mu_R(M)-1$. Choose a minimal generating set $x_1,\dots,x_{r+1}$ for $M$ such that $F:=\langle x_1,\dots,x_r \rangle$ is a free $R$-module (such a generating set always exists by \Cref{gensetexists}). Then there is an isomorphism $\theta:(F:x_{r+1}) \to \Omega^1_R(M)$. 
\end{lem}

\begin{proof}
If $a \in (F:x_{r+1})$ then $ax_{r+1} \in F$, and as $F$ is free, $ax_{r+1}$ can be written uniquely as 
\[ax_{r+1}=s^a_1x_1+\cdots +s^a_rx_r.\]
We define $\theta:(F:x_{r+1}) \to \Omega^1_R(M)$ by 
\[a \mapsto \begin{pmatrix} -s^a_1 \\ \vdots \\ -s^a_r \\ a \end{pmatrix}.\]
It is easy to see that the map $\Omega^1_R(M) \to (F:x_{r+1})$ given by projection onto the $r+1$st component is an inverse for $\theta$. 
\end{proof}

\begin{lem}\label{freesummandwhen}
Suppose $R=k[t^{a_1},\dots,t^{a_n}]$ is a numerical semigroup ring of minimal multiplicity that is not a DVR. Note that $t^{a_n+i} \in R$ since $a_n-a_1$ is the Frobenius number of $R$ (see \cite[Exercise 2.14]{RG09}). Consider a matrix $A$ of the form
\[A:=\begin{pmatrix} t^{a_n} & 0 & t^{a_n+i} \\ 0 & t^{a_1} & t^{a_2} \end{pmatrix}\]
where $0<i<a_1$ and set $M:=\im(A)$. Then $\tr_R(M)=R$ if and only if $t^{|a_2-a_1-i|} \in R$. 
\end{lem}

\begin{proof}
Set $x_1$, $x_2$, and $x_3$ to be the corresponding columns of $A$, and set $F:=\langle x_1,x_2 \rangle$. Note that $F$ is free on the basis $x_1,x_2$, and note as well that $(F:x_3)=(t^{a_n}:t^{a_n+i}) \cap (t^{a_1}:t^{a_2})$. But since $R$ has minimal multiplicity, we have $(t^{a_1}:t^{a_2})=\m$ since $a_2-a_1$ is a Pseudo-Frobenius number for $R$ (see \cite[Exercise 2.14]{RG09}). Thus $(F:M)=(F:x_3)=(t^{a_n}:t^{a_n+i})$, which is Ulrich as an $R$-module by \Cref{ulrichex} (2). Let $(t^{s_1},\dots,t^{s_n})$ be a minimal generating set for $(t^{a_n}:t^{a_n+i})$. Then via the map $\theta$ from Lemma \ref{rankoneoff}, we see that $\Omega^1_R(M)=\im(B)$ where
\[B:=\begin{pmatrix} -t^{s_1+i} & -t^{s_2+i} & \cdots & -t^{s_n+i} \\ -t^{a_2+s_1-a_1} & -t^{a_2+s_2-a_1} & \cdots & -t^{a_2+s_n-a_1} \\ t^{s_1} & t^{s_2} & \cdots & t^{s_n} \end{pmatrix}.\]
Then $R$ is a summand of $M$ if and only if elementary row operations on $B$ produce a row of zeroes. But by degree considerations this is the case if and only if $t^{a_2+s_i-a_1-s_i}=t^{a_2-a_1} \in R$, $t^{s_i+i-s_i}=t^{i} \in R$, or $t^{|a_2+s_i-a_1-(s_i+i)|}=t^{|a_2-a_1-i|} \in R$. But $i \notin R$ since $0<i<a_1$, while $a_2-a_1 \notin R$ since it is a Psuedo-Frobenius number, as noted above. Therefore, $R$ is a summand of $M$ if and only if $t^{|a_2-a_1-i|} \in R$, as desired. 
\end{proof}

\begin{theorem}\label{syztracemimpliesconductorm}
Suppose $R$ is a numerical semigroup ring of minimal multiplicity that is not a DVR. Then the following are equivalent:
\begin{enumerate}
\item[$(1)$] For any torsion-free $R$-module $M$, if $\tr_R(M)=\m$, then $\tr_R(\Omega^1_R(M))=\m$.
\item[$(2)$] For any torsion-free $R$-module $M$ with $\mu_R(M)=3$ and $\rank_R(M)=2$, if $\tr_R(M)=\m$, then $\tr_R(\Omega^1_R(M))=\m$.
\item[$(3)$] Either $e(R)=2$ or $\mathfrak{c}_R=\m$, that is, $R=k[t^{a_1},t^{a_1+1},\dots,t^{2a_1-1}]$.
\item[$(4)$] For any torsion-free $R$-module $M$ with no nonzero free summand, we have $\tr_R(M) \subseteq \tr_R(\Omega^1_R(M))$

\end{enumerate}
\end{theorem}

\begin{proof}
We note $(1) \Rightarrow (2)$ is clear. For $(3) \Rightarrow (1)$ the case where $\mathfrak{c}_R=\m$ follows as every trace ideal contains $\mathfrak{c}_R$ (see e.g. \cite[Corollary 3.6]{DS23}). Suppose on the other hand that $a_1=2$. Then if $M$ is a torsion-free $R$-module, it follows from \cite[Theorem 4.18]{LW12} that $M \cong \bigoplus^m_{j=1} I_j$ where each $I_j$ is a two-generated ideal of $R$. If $\tr_R(M)=\sum^m_{i=1} \tr_R(I_j)=\m$, we must have $\tr_R(I_j)$ contains a minimal reduction of $\m$ for some $j$, and \cite[Proposition 3.18]{lyle2024annihilators} gives that $\tr_R(I_j)=\m$.
But as $I_j$ is two-generated, we have $\Omega^1_R(I_j) \cong I_j^*$ from \cite[Lemma 3.3]{HH05}, and then we have $\m \subseteq \tr_R(I_j^*) \subseteq \tr_R(\Omega^1_R(M))$. As $\tr_R(\Omega^1_R(M)) \ne R$ from \cite[Corollary 1.2.5]{Av10}, it follows from \Cref{traceprop} (5) that $\tr_R(\Omega^1_R(M))=\m$

We now show $(2) \Rightarrow (3)$ via contrapositive; suppose $a_1>2$ and that $\mathfrak{c}_R \ne \m$. Then there exists $0<i<a_1$ with $a_1+i \notin R$. For any such $i$, consider the matrix 
\[A_i:=\begin{pmatrix} t^{a_n} & 0 & t^{a_n+i} \\ 0 & t^{a_1} & t^{a_2} \end{pmatrix}\]
and set $M_i:=\im(A_i)$. Let $F=\im\begin{pmatrix} t^{a_n} & 0\\ 0 & t^{a_1} \end{pmatrix}$.

We claim $t^{a_1} \notin (t^{a_n}:t^{a_n+i})$. Indeed, if $t^{a_1} \in (t^{a_n}:t^{a_n+i})$, then there is a $c \in R$ with $t^{a_n+i+a_1}=ct^{a_n}$. But then $|c|=a_1+i$, which cannot be since $t^{a_1+i} \notin R$. In particular, $(t^{a_n}:t^{a_n+i}) \ne \m$, and it follows from Lemma \ref{colonidealtrace} that $\tr_R(t^{a_n}:t^{a_n+i}) \ne \m$. From Lemma \ref{rankoneoff}, we see that $\Omega^1_R(M_i) \cong (F:M_i)=(t^{a_n}:t^{a_n+i}) \cap (t^{a_1}:t^{a_2})$. But as in the proof of Lemma \ref{freesummandwhen}, $(t^{a_1}:t^{a_2})=\m$ since $a_2-a_1$ is a Pseudo-Frobenius number for $R$ (see \cite[Exercise 2.14]{RG09}), so combining the above, we see that $\Omega^1_R(M_i) \cong (t^{a_n}:t^{a_n+i})$, and in particular that $\tr_R(\Omega^1_R(M_i)) \ne \m$. 

On the other hand, if $f:\Omega^1_R(M_i) \to R$ is given by projection onto the second component, we have $f\begin{pmatrix} 0 \\ t^{a_1} \end{pmatrix}=t^{a_1}$, so $t^{a_1} \in \tr_R(M_i)$, and it follows from \cite[Proposition 3.18]{lyle2024annihilators} that $\m \subseteq \tr_R(M_i)$.

Now if $t^{a_1+1} \in R$, then $a_1+1=a_2$ so $a_2-a_1=1$ is a Psuedo-Frobenius number for $R$. But this means $a_1+i \in R$ for all $i$, contradicting our hypothesis that $\mathfrak{c}_R \ne \m$.

Otherwise, $t^{a_1+1} \notin R$, and we may consider $M_i$ when $i=1$. If $t^{|(a_n+1-a_n)-(a_2-a_1)|}=t^{a_2-a_1-1} \notin R$, then $\tr_R(M_1)=\m$ by \Cref{freesummandwhen}. If $t^{a_2-a_1-1} \in R$ then as $a_2-a_1-1<a_2$, we must have $a_2-a_1-1=ca_1$ for some $1<c$. Then $a_2=(c+1)a_1+1 \ge 2a_1+1$. In particular, $t^{a_1+i} \notin R$ for any $1 \le i<a_1$. As $a_1>2$, we have in particular that $t^{a_1+2} \notin R$, and in considering $M_2$, if $t^{|(a_n+2-a_n)-(a_2-a_1)|}=t^{a_2-a_1-2} \in R$, then as $0<a_2-a_1-2<a_2$, it must be that $a_2-a_1-2=da_1$ for some $d>0$. But $a_2 \equiv 1 \mod a_1$ from above, so $a_2-a_1-2 \equiv -1 \mod a_1$, so $t^{a_2-a_1-2} \notin R$, and it follows from Lemma \ref{freesummandwhen} that $\tr_R(M_2) \ne R$. Thus $\tr_R(M_2)=\m$, so the condition of $(2)$ does not hold. We have therefore shown that conditions $(1)-(3)$ are equivalent, and it remains only to consider their equivalence with $(4)$.

For $(4) \Rightarrow (1)$, if $\tr_R(M)=\m$, then in particular $M$ has no nonzero free summand, and the condition of $(4)$ forces $\m \subseteq \tr_R(\Omega^1_R(M))$. But from \cite[Corollary 1.2.5]{Av10}, $\Omega^1_R(M)$ cannot have a nonzero free summand, and so $\tr_R(\Omega^1_R(M))=\m$.

Finally, we show $(3) \Rightarrow (4)$. If $e(R)=2$ and if $M$ is a torsion-free $R$-module then we appeal again to \cite[Theorem 4.18]{LW12} to get that $M \cong \bigoplus^m_{j=1} I_j$ where each $I_j$ is a two-generated ideal of $R$. In particular, from \cite[Lemma 3.3]{HH05}, we have $\Omega^1_R(M) \cong M^*$ and then $\tr_R(M) \subseteq \tr_R(\Omega^1_R(M))$ from \Cref{traceprop} (4). If instead we have $\mathfrak{c}_R=\m$, then the only nonzero trace ideals in $R$ are $\mathfrak{c}_R$ and $R$ from \Cref{traceprop} (9) and (10). So if $M$ is a torsion-free $R$-module without nonzero summand, then $\tr_R(M)=\mathfrak{c}_R$, and then $\tr_R(M) \subseteq \tr_R(\Omega^1_R(M))$ from \Cref{traceprop} (9), completing the proof.

\end{proof}

Combining the results of this section yields the following key consequence which encompasses Theorem \ref{introthm2}:

\begin{theorem}\label{mainconductortm}
Suppose $R=k[t^{a_1},\dots,t^{a_n}]$ is a numerical semigroup ring. Then the following are equivalent:
\begin{enumerate}
\item[$(1)$] For any torsion-free $R$-module $M$ with no nonzero free summand, we have $\tr_R(M) \subseteq \tr_R(\Omega^1_R(M))$.
\item[$(2)$] For any torsion-free $R$-module $M$ with $\rank(M)=\mu_(M)-1$ and with no nonzero free summand, we have $\tr_R(M) \subseteq \tr_R(\Omega^1_R(M))$.
\item[$(1)$] For any torsion-free $R$-module $M$, if $\tr_R(M)=\m$, then $\tr_R(\Omega^1_R(M))=\m$.
\item[$(2)$] Either $e(R)=2$ or $\mathfrak{c}_R=\m$.
\end{enumerate}
\end{theorem}

\begin{proof}
Note that $(2) \Rightarrow (1)$ is immediate from \Cref{syztracemimpliesconductorm}, while $(1) \Rightarrow (2)$ follows from combining \Cref{syztracemimpliesconductorm} with \Cref{Rhasminmult}.
\end{proof}

\section{{On the trace of the syzygy of the canonical module $\Omega^1_R(\mathfrak{c}_R)$}} \label{conductortrace}

\Cref{mainconductortm} shows that the condition $\tr_R(M)=\m$ will only force $\tr_R(\Omega^1_R(M))=\m$, i.e., to be as large as possible under exceedingly stringent hypotheses. It is natural to wonder about the opposite extreme. That is, we may consider $\tr_R(\Omega^1_R(\mathfrak{c}_R))$ in particular as an invariant of $R$, and then ask when we will have $\tr_R(\Omega^1_R(\mathfrak{c}_R))=\mathfrak{c}_R$. However, at least for numerical semigroup rings it turns out that this can also occur under similarly restrictive hypotheses. To provide this characterization, we study the structure of $\tr_R(\Omega^1_R(I))$, beginning with the following lemma:

\begin{lem}\label{mingenforulrich}
Suppose $R$ is CM with $\dim(R)=1$ and suppose $I$ is an ideal of positive grade in $R$ that is Ulrich as an $R$-module. Suppose $y$ is a minimal reduction for $\m$, and let $s$ be an integer for which $y^s \in I$ but for which $y^{s-1} \notin I$ (such an integer always exists since $I$ is $\m$-primary). Then $y^s$ is part of a minimal generating set for $I$. 
\end{lem}

\begin{proof}
Note that $y$ is a parameter of $R$, and is thus a nonzerodivisor since $R$ is CM. If $t^s \in \m I$, then since $I$ is Ulrich, we have $y^s \in \m I=yI$. So $y^s=ya$ for some $a \in I$, but as $y$ is a nonzerodivisor, this forces $y^{s-1}=a \in I$, which contradicts the choice of $s$. Thus, $y^s \in I-\m I$ and so may be extended to a minimal generating set for $I$.
\end{proof}

\begin{prop}\label{specialideals}
Let $R=k[t^{a_1},\dots,t^{a_n}]$ be a numerical semigroup ring. If for some $s \ge 1$, the ideal $(t^{sa_1},t^{a_2},\dots,t^{a_n})$ is trace, then $(t^{ua_1},t^{a_2},\dots,t^{a_n})$ is trace for all $1 \le u \le s$.
\end{prop}

\begin{proof}
It suffices to show that if $I=(t^{sa_1},t^{a_2},\dots,t^{a_n})$ is not trace, then $J=(t^{(s+1)a_1},t^{a_2},\dots,t^{a_n})$ is not trace. If $t^{sa_1} \in (t^{a_2},\dots,t^{a_n})$, then $I=J$, and there is nothing to prove. So we may suppose this is not the case, and that in particular $t^{sa_1}$ is a minimal generator for $I$. 

If $I$ is not trace, then as $R/I$ is a principal ideal ring with maximal ideal generated by $t^{a_1}$, it follows that there is a minimal generator $t^{va_1} \in \tr_R(I)$ for some $v<s$. By \Cref{traceprop} (7), we have $t^{va_1} \in (R:_{Q(R)}I)I$, and so there are homogeneous minimal generators $t^j$ for $(R:_{Q(R)} I)$ and $t^a$ for $I$ with $t^jt^a=t^{a+j}=t^{va_1}$. If $t^a=t^{sa_1}$, then we have $j=(v-s)a_1$. But by assumption $t^{a_2+j}=t^{a_2+(v-s)a_1} \in R$. Since $v<s$, this would force $t^{a_2}=t^{a_2+(v-s)a_1}t^{(s-v)a_1} \in (t^{a_1})$, which cannot be. Thus we must have $a=a_i$ for some $2 \le i \le n$. But then $t^j \in (R:_{Q(R)} J)$ and we have $t^{va_1} \in (R:_{Q(R)} J)J=\tr_R(J)$, so $J$ is not trace.
\end{proof}

The following example shows that taking trace ideals does not preserve inclusion even amongst ideals of the form consider in \Cref{specialideals}. 

\begin{example}
Let $R=k[t^3,t^7,t^{11}]$ and consider the ideals $I=(t^{12},t^7,t^{11})$ and $J=(t^{15},t^7,t^{11})$. Then $\tr_R(I)=(t^6,t^7,t^8)$ and $\tr_R(J)=\m$. In particular, we note that $J \subseteq I$, but $\tr_R(J) \nsubseteq \tr_R(I)$.
\end{example}

\begin{lem}\label{prelimlem}
Suppose $R=k[t^{a_1},\dots,t^{a_n}]$ is a numerical semigroup ring and let $F$ be the Frobenius number for $R$. Then for any $i \ne j$, we have
\[(t^{F+i}:t^{F+j})=\begin{cases} (t^{F+1}:t^{F+j-i+1}) & \mbox{if } j>i \\ (t^{F+i-j+1)}:t^{F+1)} & \mbox{if } j<i \end{cases}.\] Moreover, we have 
\[I_1(A)=\sum^{a_1}_{j=1} \sum_{i \ne j} (t^{F+i}:t^{F+j})=\sum^{a_1}_{i=2} \tr_R(t^{F+1},t^{F+i}) \subseteq \sum^{a_1}_{i=2} \tr_R(t^{F+1}:t^{F+i})=\sum^{a_1}_{j=2} \sum_{i<j} \tr_R(t^{F+i}:t^{F+j})\]
where $A$ is a minimal homogeneous presentation matrix for $\mathfrak{c}_R$.
\end{lem}

\begin{proof}
Observe that a homogeneous element $t^a \in R$ is in $(t^{F+i}:t^{F+j})$ if and only if $t^{a+F+j-(F+i)}=t^{a+j-i} \in R$. If $i<j$, then this holds if and only if $t^{a+j-i+1-1} \in R$, that is, if $t^a \in (t^{F+1}:t^{F+j-i+1})$. On the other hand, if $j<i$, then similarly we have $t^a \in (t^{F+i}:t^{F+j})$ if and only if $t^a \in (t^{F+i-j+1}:t^{F+1})$. 

It follows from combining the above with \Cref{traceprop} (8) and \Cref{containedinsumtrace} that 
\[I_1(A)=\sum^{a_1}_{j=1} \sum_{i \ne j} (t^{F+i}:t^{F+j})=\sum^{a_1}_{i=2} ((t^{F+1}:t^{F+i})+(t^{F+i}:t^{F+1}))\]
\[= \sum^{a_1}_{i=2} \tr_R(t^{F+1},t^{F+i}) \subseteq \sum^{a_1}_{i=2} \tr_R(t^{F+1}:t^{F+i})=\sum^{a_1}_{j=2} \sum_{i<j} \tr_R(t^{F+i}:t^{F+j})\]
as claimed.
\end{proof}

The next theorem gives a description of $\tr_R(\Omega^1_R(\mathfrak{c}_R))$ and in particular demonstrates that the containments of \Cref{containedinsumtrace} can be equalities without the reflexivity condition imposed on the two-generated ideals.

\begin{theorem}\label{conductorsyzthm}
Suppose $R=k[t^{a_1},\dots,t^{a_n}]$ is a numerical semigroup ring and let $F$ denote the Frobenius number for $R$. If $A$ is a minimal homogeneous presentation matrix for $\mathfrak{c}_R$, then $\tr_R(\Omega^1_R(\mathfrak{c}_R))=I_1(A)=\sum^{a_1}_{i=2} \tr_R(t^{F+1},t^{F+i})=\sum^{a_1}_{i=2} \tr_R(t^{F+1}:t^{F+i})$. 
\end{theorem}

\begin{proof}

We claim that $\tr_R(t^{F+1}:t^{F+i}) \subseteq \sum^{a_1}_{i=2} ((t^{F+1}:t^{F+i})+(t^{F+i}:t^{F+1}))$ for all $i$, which combined with \Cref{prelimlem} will complete the proof.

We note from \Cref{conductorfact} that $\mathfrak{c}_R=(t^{F+1}:\mathfrak{c}_R)=\bigcap_{i=2}^{a_1} (t^{F+1}:t^{F+i})$. In particular, $\mathfrak{c}_R \subseteq (t^{F+1}:t^{F+i})$ for all $2 \le i \le a_1$. From \Cref{ulrichex} (3) we have that $\mathfrak{c}_R$ is Ulrich, and then \Cref{tracesyzylrich} gives that $t^{a_2},\dots,t^{a_n} \in \sum^{a_1}_{j=2} \tr_R(t^{F+1},t^{F+j})$, so it suffices to show that if $t^{ua_1} \in \tr_R(t^{F+1}:t^{F+j})$ for some $2 \le j \le a_1$, then $t^{ua_1} \in (t^{F+1}:t^{F+i})$ or $(t^{F+i}:t^{F+1})$ for some $i$. We may moreover suppose that $t^{ua_1}$ is a minimal generator for $\tr_R(t^{F+1}:t^{F+j})$.

If $t^{ua_1} \in \tr_R(t^{F+1}:t^{F+j})$, then from \Cref{traceprop} (7), there are homogeneous generators $t^{\ell}$ for $(R:_{Q(R)} (t^{F+1}:t^{F+j}))$ and $t^a$ for $(t^{F+1}:t^{F+j})$ with $t^{a+\ell}=t^{ua_1}$. As $\mathfrak{c}_R \subseteq (t^{F+1}:t^{F+j})$, we have $t^{\ell} \in (R:_{Q(R)} \mathfrak{c}_R)=(R:_{Q(R)} R:_{Q(R)} \bar{R})=\bar{R}$. In particular, $\ell \ge 0$, and we may suppose $\ell>0$. Write $\ell=va_1+k$ where $0 \le k<a_1$, and note $0 \le v<u$. Then $t^a=t^{(u-v)a_1-k} \in (t^{F+1}:t^{F+j})$, so $t^{(u-v)a_1+j-k-1} \in R$. 

We consider several cases:
\begin{enumerate}
\item If $j>k+1$, then $t^{(u-v)a_1} \in (t^{F+1}:t^{F+j-k})$.

\item If $j=k+1$, then $t^{(u-v)a_1-k}=t^{(u-v)a_1-j+1} \in R$ which forces $t^{(u-v)a_1} \in (t^{F+j}:t^{F+1})$.

\item Finally, if $j \le k$, then we note that $a_1+j-k>2$. Indeed, $j \ge 2$ while $k<a_1$, so $j-k>2-a_1$. Then as $t^{(u-v-1)a_1+a_1+j-k-1} \in R$ we have $t^{(u-v-1)a_1} \in (t^{F+1}:t^{F+a_1+j-k})$.
\end{enumerate}

In particular, in every possible case we have $t^{ua_1} \in (t^{F+1}:t^{F+i})+(t^{F+i}:t^{F+1})$ for some $i$, completing the proof.
\end{proof}

We give the following example to note that one cannot, in the context of \Cref{conductorsyzthm}, expect to have $\tr_R(t^{F+1},t^{F+i})=\tr_R(t^{F+1}:t^{F+i})$ in general:

\begin{example}
Let $R=k[t^7,t^9,t^{11},t^{13}]$. The Frobenius number for $R$ is $19$. As in the proof of \Cref{conductorsyzthm}, we have $t^a \in (t^{20}:t^{21})$ if and only if $t^a \in R$ and $t^{a+1} \in R$, while $t^a \in (t^{21}:t^{20})$ if and only if $t^a \in R$ and $t^{a-1} \in R$. We thus see that $(t^{20}:t^{21})=(t^{13},t^{21},t^{23},t^{25})$ and $(t^{21}:t^{20})=(t^{14},t^{22},t^{24},t^{26})$, and in particular, we have $\tr_R(t^{20},t^{21})=(t^{20}:t^{21})+(t^{21}:t^{20})=(t^{13},t^{14})$. But we note that as long as $a \ge 17$, then $16+a \ge 32=13+20$. In particular, as the Frobenius number of $R$ is $19$, it follows that $t^{16} \in (t^{13}:(t^{21},t^{23},t^{25})=(t^{13}:(t^{20}:t^{21}))$. In particular, $t^{16} \in \tr_R(t^{20}:t^{21})$ by \Cref{traceprop} (8), so $\tr_R(t^{20},t^{21}) \ne \tr_R(t^{20}:t^{21})$. 
\end{example}

\begin{theorem}\label{syzofconductorthm}
Suppose $R=k[t^{a_1},\dots,t^{a_n}]$ is a numerical semigroup ring. Then the following are equivalent:
\begin{enumerate}
\item[$(1)$] $\tr_R(\Omega^1_R(\mathfrak{c}_R))=\mathfrak{c}_R$.
\item[$(2)$] $R$ is not a DVR and for some integer $s$, we have $R=k[t^{a_1},t^{sa_1+1},t^{sa_1+2},\dots,t^{sa_1+a_1-1}]$
\end{enumerate}
\end{theorem}

\begin{proof}
We first show $(2) \Rightarrow (1)$. If $R=k[t^{a_1},t^{sa_1+1},t^{sa_1+2},\dots,t^{sa_1+a_1-1}]$ then the Frobenius number associated to $R$ is $sa_1-1$, so $\mathfrak{c}_R=(t^{sa_1},t^{sa_1+1},\dots,t^{sa_1+a_1-1})$ with this generating set being minimal. We claim $(t^{sa_1}:t^{sa_1+i})=\mathfrak{c}_R$ for each $i$. Indeed, from \Cref{conductorfact}, we have $\mathfrak{c}_R=(t^{sa_1}:\mathfrak{c}_R)=\bigcap^{a_1-1}_{i=1} (t^{sa_1}:t^{sa_1+i})$, so $\mathfrak{c}_R \subseteq (t^{sa_1}:t^{sa_1+i})$. 

Suppose $t^b \in (t^{sa_1}:t^{sa_1+i})$. If $t^b \in (t^{sa_1+1},t^{sa_1+2},\dots,t^{sa_1+a_1-1})$, then $t^b \in \mathfrak{c}_R$ and there is nothing to prove. So we may suppose $b=ua_1$ for some $u$.
Then as $t^b \in (t^{sa_1}:t^{sa_1+i})$, we have $t^{b+i}=t^{ua_1+i} \in R$. If $u<s$, then as $i<a_1$, we have $ua_1+i<sa_1+1$, so $ua_1+i \equiv 0 \mod a_1$. But this cannot be, since $1 \le i \le a_1-1$. It follows that $u \ge s$, so $t^b \in \mathfrak{c}_R$. Thus $(t^{sa_1}:t^{sa_1+i})=\mathfrak{c}_R$ for each $1 \le i \le a_1-1$. But then from \Cref{conductorsyzthm} and from \Cref{traceprop} (1) and (10) we have \[\tr_R(\Omega^1_R(\mathfrak{c}_R))=\sum^{a_1-1}_{i=1} \tr_R(t^{sa_1}:t^{sa_1+i})=\sum^{a_1-1}_{i=1} \tr_R(\mathfrak{c}_R)=\sum^{a_1-1}_{i=1} \mathfrak{c}_R=\mathfrak{c}_R.\]

We now show $(1) \Rightarrow (2)$. Suppose $\tr_R(\mathfrak{c}_R)=\mathfrak{c}_R$. Note $R$ cannot be a DVR since then $\mathfrak{c}_R=R$ and then $\Omega^1_R(\mathfrak{c}_R)$ would be $0$, so $\tr_R(\Omega^1_R(\mathfrak{c}_R))=0 \ne \mathfrak{c}_R=R$. We may thus suppose that $a_1>1$. By \Cref{tracesyzylrich}, there is an $s$ for which $\tr_R(\Omega^1_R(\mathfrak{c}_R))=(t^{sa_1},t^{a_2},\dots,t^{a_n})$, in particular with $t^{(s-1)a_1)} \notin \tr_R(\Omega^1_R(\mathfrak{c}_R))$. So we have $\mathfrak{c}_R=(t^{sa_1},t^{a_2},\dots,t^{a_n})$, which by \Cref{ulrichex} (2) forces $(t^{sa_1},t^{a_2},\dots,t^{a_n})$ to be Ulrich. By \Cref{mingenforulrich}, $t^{sa_1},t^{a_2},\dots,t^{a_n}$ form a minimal generating set for $\mathfrak{c}_R$. In particular, $sa_1=F+i$ for some $1 \le i \le a_1$, where $F$ denotes the Frobenius number associated to $R$. 

Consider first the case where $i>1$.
Then 
\[t^{(s-1)a_1}t^{F+a_1}=t^{(s-1)a_1+F+a_1}=t^{sa_1+F}=t^{2F+i}=t^{F+i-1}t^{F+1}.\]
Thus $t^{(s-1)a_1} \in (t^{F+1}:t^{F+a_1}) \subseteq (\sum_{i=1}^{a_1-1} t^{F+i}:t^{F+a_1})$. But combining \Cref{I1colon} with \Cref{idealsyztrace} gives that $t^{(s-1)a_1} \in \tr_R(\Omega^1_R(\mathfrak{c}_R))=\mathfrak{c}_R$, a contradiction. 

It follows that $i=1$, so that $F=sa_1-1$. Then $a_i=sa_1+i-1$ for all $i \ge 2$, meaning $R=k[t^{a_1},t^{sa_1+1},t^{sa_1+2},\dots,t^{sa_1+a_1-1}]$, as desired. 
\end{proof}

\bibliography{reference}
\bibliographystyle{amsalpha}

\end{document}